\documentclass[12pt]{amsart}
\usepackage{amsfonts, amssymb, latexsym}

\newtheorem{theorem}{Theorem}[section]
\newtheorem{proposition}[theorem]{Proposition}

\newtheorem{claim}[theorem]{Claim}
\newtheorem*{claim*}{Claim}

\newtheorem{Main Conjecture}[theorem]{Main Conjecture}

\theoremstyle{remark}

\newtheorem{example}[theorem]{Example}

\theoremstyle{plain}
\newtheorem{question}{Question}

\newcommand{\cellsize}{19}
\newlength{\cellsz} 
\newsavebox{\cell}
\sbox{\cell}{\begin{picture}(\cellsize,\cellsize)
\put(0,0){\line(1,0){\cellsize}}
\put(0,0){\line(0,1){\cellsize}}
\put(\cellsize,0){\line(0,1){\cellsize}}
\put(0,\cellsize){\line(1,0){\cellsize}}
\end{picture}}
\newcommand\cellify[1]{\def\thearg{#1}\def\nothing{}%
\ifx\thearg\nothing
\vrule width0pt height\cellsz depth0pt\else
\hbox to 0pt{\usebox{\cell} \hss}\fi%
\vbox to \cellsz{
\vss
\hbox to \cellsz{\hss$#1$\hss}
\vss}}
\newcommand\tableau[1]{\vtop{\let\\\cr
\baselineskip -16000pt \lineskiplimit 16000pt \lineskip 0pt
\ialign{&\cellify{##}\cr#1\crcr}}}
\DeclareMathOperator{\wt}{{\tt wt}}

\newcommand{\gap}{\hspace{1in} \\ \vspace{-.2in}}

\newcommand{\excise}[1]{}

\begin{document}
\pagestyle{plain}

\title{Vanishing of Littlewood-Richardson polynomials is in ${\sf P}$}
\author{Anshul Adve}
\author{Colleen Robichaux}
\author{Alexander Yong}
\address{Dept.~of Mathematics, U.~Illinois at Urbana-Champaign, Urbana, IL 61801, USA} 
\email{anshul.adve@gmail.com, cer2@illinois.edu, ayong@uiuc.edu}
\date{August 14, 2017}

\begin{abstract}
J.~DeLoera-T.~McAllister and K.~D.~Mulmuley-H.~Narayanan-M.~Sohoni independently
proved that determining the vanishing of Littlewood-Richardson coefficients has strongly polynomial time computational complexity. Viewing these as Schubert calculus numbers, we prove the generalization to the Littlewood-Richardson polynomials that control equivariant cohomology of Grassmannians. We construct a polytope using the edge-labeled tableau rule of 
H.~Thomas-A.~Yong. Our proof then  combines a saturation theorem of D.~Anderson-E.~Richmond-A.~Yong, a reading order independence property, and \'E.~Tardos' algorithm for combinatorial linear programming.
\end{abstract}

\maketitle
\section{Introduction}

Let $\lambda=(\lambda_1\geq \lambda_2\geq \cdots \geq \lambda_n\geq 0)$ be a partition with $n$ nonnegative parts.
We identify it in the usual manner with its Ferrers/Young diagram, where the 
$i$-th row consists of $\lambda_i$ boxes. Consider
a grid with $n$ rows and $m \geq n + \lambda_1 - 1$ columns. Place $\lambda$ in the northwest corner; this is the
{\bf initial diagram} for $\lambda$. 

For example,
if $\lambda=(4,2,1,0,0)$, the
initial diagram is the first
of the three below.
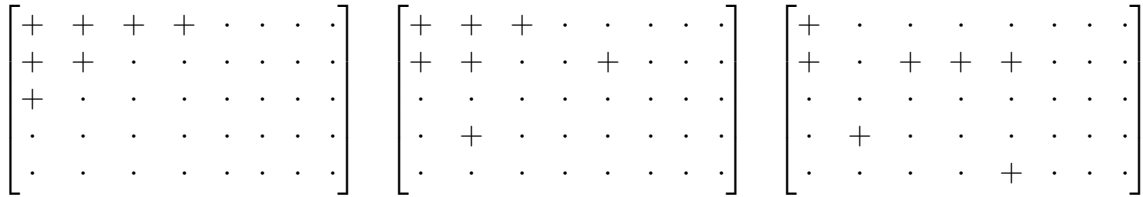
\[\left[\begin{matrix}
+ & + & + & + & \cdot & \cdot & \cdot & \cdot \\
+ & + & \cdot & \cdot & \cdot & \cdot & \cdot & \cdot \\
+ & \cdot & \cdot & \cdot & \cdot & \cdot & \cdot & \cdot \\
\cdot & \cdot & \cdot & \cdot & \cdot & \cdot & \cdot & \cdot \\
\cdot & \cdot & \cdot & \cdot & \cdot & \cdot & \cdot & \cdot  
\end{matrix}\right] \ \ \ 
\left[\begin{matrix}
+ & + & + & \cdot & \cdot & \cdot & \cdot & \cdot \\
+ & + & \cdot & \cdot & + & \cdot & \cdot & \cdot \\
\cdot & \cdot & \cdot & \cdot & \cdot & \cdot & \cdot & \cdot \\
\cdot & + & \cdot & \cdot & \cdot & \cdot & \cdot & \cdot \\
\cdot & \cdot & \cdot & \cdot & \cdot & \cdot & \cdot & \cdot  
\end{matrix}\right] \ \ \ 
\left[\begin{matrix}
+ & \cdot & \cdot & \cdot & \cdot & \cdot & \cdot & \cdot \\
+ & \cdot & + & + & + & \cdot & \cdot & \cdot \\
\cdot & \cdot & \cdot & \cdot & \cdot & \cdot & \cdot & \cdot \\
\cdot & + & \cdot & \cdot & \cdot & \cdot & \cdot & \cdot \\
\cdot & \cdot & \cdot & \cdot & + & \cdot & \cdot & \cdot  
\end{matrix}\right]
\]
A {\bf local move} is a mutation of any $2\times 2$ subsquare of the form
\[
\begin{matrix}
+ & \cdot\\
\cdot & \cdot
\end{matrix}\ \ \  \ \ 
\mapsto  \ \ \ \ \ 
\begin{matrix}
\cdot & \cdot\\
\cdot & +
\end{matrix}\]
A {\bf plus diagram} is a configuration of $+$'s in the grid resulting
from successive applications of the local move to the initial diagram for $\lambda$. Above one sees two more of the many other plus diagrams for $\lambda=(4,2,1,0,0)$.

Let ${\tt Plus}(\lambda)$ be the set of plus diagrams for $\lambda$. Given $P\in {\tt Plus}(\lambda)$, let ${\wt}_x(P)$ be the monomial $x_1^{\alpha_1} x_2^{\alpha_2} \cdots x_n^{\alpha_n}$ where $\alpha_i$ is the number of $+$'s in row $i$ of $P$. A finer statistic is 
\[{\wt }_{x,y}(P)=\prod_{(i,j)}x_i-y_j,\] 
where the product is over all $(i,j)$ such that there is a $+$ in row $i$ and column $j$ of $P$.  For example, if $P$ is the rightmost diagram above, then
\[{\tt wt}_x(P)=x_1 x_2^4 x_4 x_5 \text{ \and \ } {\tt wt}_{x,y}(P)= 
(x_1-y_1)(x_2-y_1)(x_2-y_3)(x_2-y_4)(x_2-y_5)(x_4-y_2)(x_5-y_5).\]

Let $X=\{x_1,x_2,\ldots,x_n\}$ and $Y=\{y_1,y_2,\ldots\}$ be two collections of indeterminates. We consider two
generating series, the second being a refinement of the first:
\[s_{\lambda}(X)=\sum_{P\in {\tt Plus}(\lambda)}  {\wt}_x(P) \text{ \ \ and \ \ } s_{\lambda}(X; Y)=\sum_{P\in {\tt Plus}(\lambda)}  {\wt}_{x,y}(P).\]
These are the {\bf Schur polynomial} and {\bf factorial Schur polynomial}, respectively. A more standard description of these polynomials involves semistandard Young tableaux, see, e.g., \cite{Macdonald} and the references therein. The description above arises in, e.g., \cite{KMY}.

Let ${\sf Sym}[X]$ denote the ring of symmetric polynomials in $X$. Then the set of Schur polynomials $s_{\lambda}(X)$
(over partitions $\lambda$ with at most $n$ rows) is a ${\mathbb Z}$-linear basis of ${\sf Sym}[X]$. Analogously, the factorial Schur polynomials 
form a 
${\mathbb Z}[Y]$-linear basis of ${\sf Sym}[X]\otimes_{\mathbb Q} {\mathbb Z}[Y]$.

The structure constants with respect to these bases are defined by
\[s_{\lambda}(X)s_{\mu}(X)
=\sum_{\nu}c_{\lambda,\mu}^{\nu} s_{\nu}(X) \text{\ \ and \ \ } s_{\lambda}(X; Y)s_{\mu}(X; Y)
=\sum_{\nu}C_{\lambda,\mu}^{\nu} s_{\nu}(X; Y).\]
Here, $c_{\lambda,\mu}^{\nu}$ is the {\bf Littlewood-Richardson coefficient}; this is known to be a nonnegative integer. Following the terminology of \cite{Molev}, the {\bf Littlewood-Richardson polynomial} is $C_{\lambda,\mu}^{\nu}\in {\mathbb Z}[Y]$. 
These latter coefficients
generalize the former, 
i.e., 
\[c_{\lambda,\mu}^{\nu}=C_{\lambda,\mu}^{\nu}|_{y_1=0,y_2=0,\ldots}.\]
In general, $c_{\lambda,\mu}^{\nu}=0$ unless $|\lambda|+|\mu|=|\nu|$ whereas
$C_{\lambda,\mu}^{\nu}=0$ unless $|\lambda|+|\mu|\geq |\nu|$, where here $|\lambda|=\sum_i \lambda_i$.
It is a theorem of W.~Graham \cite{Graham}
that $C_{\lambda,\mu}^{\nu}$ is uniquely expressible as a polynomial, with nonnegative integer coefficients in the variables $\{\beta_i:=y_{i+1}-y_i:i\geq 1\}$.

For example, the interested reader may verify that
\[s_{(1,0)}(x_1,x_2;Y)^2=s_{(2,0)}(x_1,x_2;Y)
+s_{(1,1)}(x_1,x_2;Y)+(y_3-y_2)s_{(1,0)}(x_1,x_2;Y).\]

J.~DeLoera-T.~McAllister \cite{DeLoera} and
K.~D.~Mulmuley-H.~Narayanan-M.~Sohoni \cite{MuNaSo12} independently proved that the vanishing problem for $c_{\lambda,\mu}^{\nu}$ has strongly polynomial time complexity. The following result completes the parallel above:

\begin{theorem}
\label{thm:main}
The vanishing of $C_{\lambda,\mu}^{\nu}$ can be decided in strongly polynomial time.
\end{theorem}

In contrast, H.~Narayanan \cite{Nara} has shown that computation of $c_{\lambda,\mu}^{\nu}$ is a $\#{\sf P}$-complete problem in L.~Valiant's complexity theory for counting problems \cite{Valiant}. 
Since $c_{\lambda,\mu}^{\nu}$ is a special case of $C_{\lambda,\mu}^{\nu}$, it follows from any of the combinatorial rules in Section~\ref{section:rule} that determining the value of $C_{\lambda,\mu}^{\nu}|_{\beta_i=1}\in {\mathbb Z}_{\geq 0}$ is also $\#{\sf P}$-complete. In particular, no polynomial time algorithm for either counting problem  can exist unless ${\sf P}={\sf NP}$.

Our proof is a modification of the argument used in \cite{DeLoera, MuNaSo12}. In Section~\ref{section:mainProof}, we construct, with explicit inequalities, a polytope ${\mathcal P}_{\lambda,\mu}^{\nu}$ with the property that 
${\mathcal P}_{\lambda,\mu}^{\nu}$ has a lattice point if and only if 
$C_{\lambda,\mu}^{\nu}\neq 0$. Now, if ${\mathcal P}_{\lambda,\mu}^{\nu}$ is nonempty,
it has a rational vertex. In that case, some dilation $N{\mathcal P}_{\lambda,\mu}^{\nu}$ contains an integer lattice point. Moreover, by our construction, 
$N{\mathcal P}_{\lambda,\mu}^{\nu}={\mathcal P}_{N\lambda,N\mu}^{N\nu}$, 
which means $C_{N\lambda,N\mu}^{N\nu}\neq 0$. Thus, by a saturation theorem of D.~Anderson, E.~Richmond, and the third author
\cite{AnRiYo13}, we conclude
\[C_{\lambda,\mu}^{\nu}\neq 0\iff C_{N\lambda,N\mu}^{N\nu}\neq 0 \iff
{\mathcal P}_{\lambda,\mu}^{\nu}\neq \emptyset.\] 
To determine if ${\mathcal P}_{\lambda,\mu}^{\nu}\neq \emptyset$, one 
needs to ascertain feasiblity of any linear programming problem involving ${\mathcal P}_{\lambda,\mu}^{\nu}$. Famously, the Klee-Minty cube shows that the practically efficient simplex method has
exponential worst-case complexity. Instead, one can
appeal to  ellipsoid/interior point methods for polynomiality. Better yet,
our inequalities are of the form
$A{\bf x}\leq {\bf b}$ where 
the entries of $A$ are from $\{-1,0,1\}$ and the vector ${\bf b}$ is integral.
Hence our polytope is \emph{combinatorial} and so one  can achieve
strongly polynomial time complexity using \'E.~Tardos' algorithm; see
\cite{anothertardos, Tardos}. 

We point out some aspects of our modification.
In \cite{DeLoera, MuNaSo12} the authors use the original saturation theorem 
of A.~Knutson-T.~Tao \cite{Knutson.Tao:puzzles}. In addition, the polytope used has
precisely $c_{\lambda,\mu}^{\nu}$ many lattice points. Our polytope does not have any such exact counting feature. To construct it, we need to deduce a new result about the edge-labeled tableau rule of H.~Thomas and the third author \cite{ThYo12}. The remainder of our argument is Proposition \ref{prop:finale}.

In recent years there has been significant work on the complexity of computing Kronecker coefficients; see, e.g., \cite{Burgisser, PakPanova, Iken} and the references therein.
These are an extension of the Littlewood-Richardson coefficients in the context of the representation theory of the symmetric group. This paper initiates a study of the analogous complexity issues in Schubert calculus, by interpreting the Littlewood-Richardson coefficients as triple intersections of Schubert varieties in the Grassmannian; see Section~\ref{section:Schubert} for further discussion and open problems.

\section{A factorial Littlewood-Richardson rule}\label{section:rule}

A.~Molev-B.~Sagan \cite{Molev.Sagan} gave the first combinatorial rule for $C_{\lambda,\mu}^{\nu}$. The first
rule that exhibited the positivity of \cite{Graham} was found by A.~Knutson-T.~Tao \cite{Knutson.Tao:puzzles} in terms of \emph{puzzles}. Later V.~Kreiman \cite{Kreiman} and A.~Molev \cite{Molev} independently gave essentially equivalent tableaux rules with the same positivity property. We also mention P.~Zinn-Justin's \cite{Paul} which gives a quantum integrability proof of the puzzle rule.

Actually, we will use yet another rule, due to H.~Thomas and the third author \cite{ThYo12}. This is also the rule utilized in the proof of the saturation theorem \cite{AnRiYo13} that we need. Indeed, we will observe a new property of the rule that may be of some independent interest.

\subsection{The edge-labeled rule}

We now recall the rule for $C_{\lambda,\mu}^{\nu}$ from \cite{ThYo12}.

Suppose $\lambda\subseteq \nu$. An {\bf edge-labeled tableau} $T$ of skew shape $\nu/\lambda$ and content $\mu$ is an assignment of $
\mu_i$ many labels $i$ to the boxes
of $\nu/\lambda$ and the horizontal edges weakly south of the ``southern border'' of $\lambda$ (thought of as a lattice path, in the usual way). Each box contains exactly one label. Each edge contains a (possibly empty) set of labels. 
Moreover:

\begin{itemize}
\item[(i)] the \emph{box} labels weakly increase along rows;
\item[(ii)] the labels strictly increase along columns; and
\item[(iii)] no edge label $k$ is {\bf too high}: i.e., weakly north of the the top edge of a box in row $k$. 
\end{itemize}

We will refer to (i) and (ii) as {\bf semistandardness} conditions.

A tableau is {\bf lattice} if 
for every label $k$ and column $j$, the number of $k$'s that appear in column $j$ and to the right is weakly greater than the number of $k+1$'s that appear in the same region. This can be stated in terms of a column reading word $w_{c}(T)$, obtained by reading the columns top to bottom, right to left. When reading a set-valued edge, read entries in increasing order. 

We will also need the row reading word $w_{r}(T)$. This is obtained by reading the rows right to left
and
top to bottom, and reading set-valued edges in increasing order.

We say a word is a {\bf lattice} if for every $t$ and label $k$, in reading the first $t$ letters, there are weakly more $k$ than $k+1$. 

For example, 
\[w_{c}(T)=1 \ 2 \ 1 \ 2 \ 3 \ 1 \ 3 \ 1 \ 2 \ 3 
\text{ \ \ and $w_{r}(T)=1 \ 1 \ 2 \ 2 \ 3 \ 1 \ 1 \ 3 \ 2 \ 3$}\] where

\label{exa:1.1}
\[\begin{picture}(100,50)
\put(0,35){$T=\tableau{{\ }&{\ }&{ 1 }&{1}\\{\ }&{1}\\{2}&{3}}$}
\put(31,-7){$3$}
\put(31,12){$1$}
\put(66,30){$23$}
\put(88,30){$2$}
\end{picture}\]

Notice that both words are lattice. This is the point of Theorem~\ref{theorem:readOrder} below. In addition, the edge labels on the southern border of the first row are precisely those that are too high. 

Let ${\tt EdgeTab}_{\lambda,\mu}^{\nu}$ be the set of edge labeled tableaux $T$ such that $w_c(T)$ is lattice. The main theorem of \cite{ThYo12} is that there is a weight ${\tt apwt}(T)$ such that
\[C_{\lambda,\mu}^{\nu}=\sum_{T\in {\tt EdgeTab}_{\lambda,\mu}^{\nu}} {\tt apwt}(T).\]
We do not actually need ${\tt apwt}(T)$ in this paper, so we supress this detail. Instead, to discuss nonvanishing, we only need the following immediate consequence \cite[Corollary~3.3]{AnRiYo13}:

\begin{proposition}
$C_{\lambda,\mu}^{\nu}= 0$ if and only if ${\tt EdgeTab}_{\lambda,\mu}^{\nu}=\emptyset$.
\end{proposition}

\subsection{Reading order independence}

It is well-known to experts in the theory of Young tableaux that ``any reasonable reading order works''. An instantiation of this imprecise statement is that a (classical, i.e., non-edge-labeled) semistandard tableaux is lattice for the column reading word (top to bottom, right to left) if and only if it is lattice for the row reading word (right to left, top  to bottom).

The original formulation of the rule from \cite{ThYo12} uses column reading order. However, since the saturation property concerns stretching rows, we will need the following fact:

\begin{theorem}\label{theorem:readOrder}
Let $T$ be an edge-labeled tableau. Then $w_c(T)$ is lattice if and only if $w_r(T)$ is lattice. 
\end{theorem}
\begin{proof}
Let $T$ be an edge-labeled tableau. Let 
\[T_{i,j}= \text{\ the label of the box in row $i$ column $j$ (in matrix notation).}\] Similarly,
\[T_{i+\frac{1}{2},j}= \text{the (set) filling of the southern edge of $(i,j)$.}\] 

Accordingly, we let $(x,y)$ denote either a
box or edge position of the tableau, i.e., $(x,y)=(i,j)$ or
$(x,y)=(i+\frac{1}{2},j)$.
Let 
\[w_r\mid_{(x,y)}(T)= \text{\ the row reading word of $T$ ending at $(x,y)$, and}\] 
\[w_c\mid_{(x,y)}(T)= 
\text{\ the column reading word
of $T$ ending at $(x,y)$.}\]

\begin{claim}
\gap
\begin{itemize}
\item[(I)] All labels weakly northeast of $(x,y)$ are read by both
$w_r\mid_{(x,y)}(T)$ and
$w_c\mid_{(x,y)}(T)$.
\item[(II)] If $\ell$ is read by
$w_c\mid_{(x,y)}(T)$ but not
$w_r\mid_{(x,y)}(T)$ then $\ell> T_{x,y}$.
\item[(III)] If $\ell$ is read by $w_r\mid_{(x,y)}(T)$ but
not $w_c\mid_{(x,y)}(T)$ then
$\ell< T_{x,y}$.
\end{itemize}
\end{claim}
\noindent\emph{Proof of claim:}
(I) is by definition. (II) and (III) follow since $T$ is semistandard.
\qed

$(\Rightarrow)$ Suppose $w_c(T)$ is lattice, but $w_r(T)$ is not.
Hence there exists a label $k$ and position $(x,y)$ such that
$w_r\mid_{(x,y)}$ contains more $k+1$'s than $k$'s. We may assume without loss of generality that $(x,y)$ contains $k+1$. Then by (II) and (III) of the claim, the excess of $k+1$'s must be blamed on the region weakly northeast of $(x,y)$. However, (I) then implies $w_c(T)$ is not lattice, a contradiction.

$(\Leftarrow)$ Conversely, suppose $w_r(T)$ is lattice and $w_c(T)$ is not. Take a label $k$ and position $(x,y)$
such that $w_c|_{(x,y)}(T)$ contains more $k+1$'s than $k$'s. We may assume $T_{x,y}=k+1$ and $(x,y)$ is the first (topmost and rightmost) position where such a failure occurs.

\noindent
{\sf Case $1$:} [$(x,y)=(i,j)$ is a box] By (II), among the labels read by $w_c\mid_{(i,j)}(T)$ but not $w_r\mid_{(i,j)}(T)$, no $k$ or $k+1$ appears. Therefore since $w_c\mid_{(i,j)}(T)$ is not lattice, in the region read by both, there are more $k+1$'s than $k$'s.
Since $w_r(T)$ is lattice,
in the region only read by $w_r\mid_{(i,j)}(T)$, there must exist at least one $k$. Where can such an additional $k$ appear? By semistandardness, it must be in row $i-1$, strictly to the left of column $j$, as either a box or edge label. 
Moreover, again by semistandardness,
any such extra $k$ in column $j'<j$ must have a ``paired'' $k+1$ in the box $(i,j')$ below it. Hence, it follows that $w_r(T)$ is also not lattice, a contradiction.

\noindent
{\sf Case 2:} [$(x,y)=(i+\frac{1}{2},j)$]  As in {\sf Case 1}, there must exist an extra $k$ in the region $R$ weakly north of row $i$ and strictly west of column $j$. 
Now, if there is a box label $k$ in $R$, then by semistandardness we conclude $T_{i,j}=k$. This implies $(x,y)$ is not the first violation of latticeness for $w_c(T)$.    
\end{proof}

\section{Proof of the main theorem} \label{section:mainProof}

Suppose $T\in{\tt EdgeTab}_{\lambda,\mu}^{\nu}$. 
Let $r_k^i(T)$ denote the number of $k'$s in the $i$th row of $T$ and $r_{k}^{i+\frac{1}{2}}(T)$ the number of $k$'s in the southern edges of the $i$th row of $T$, where \[k \in \{1,2,\dots,l(\mu)\} \text{\  and $i \in \{1,2,\dots,l(\nu)\}$.}\]
Recall, $l(\mu)$ is the number of nonzero parts of $\mu$, etc. By convention, let 
\[r_{l(\mu)+1}^i = r_k^{l(\nu)+1} = 0.\] 
\begin{example}
For instance, consider
\[\begin{picture}(100,50)
\put(0,35){$T=\tableau{{\ }&{\ }&{ \ }&{ \ }\\{\ }&{\ }&{1}&{1}\\{\ }&{2}&{3}}$}
\put(50,-7){$3$}
\put(50,12){$1$}
\put(69,12){$2$}
\end{picture}\]
Then 
\[r_1^2=2,r_1^{2+\frac{1}{2}}=1, r_2^{2+\frac{1}{2}}=1,
r_2^3=1, r_3^3=1,
r_3^{3+\frac{1}{2}}=1\]
and all other values are zero.\qed 
\end{example}

Next, examine the following conditions (which modify those of a preprint version of \cite{MuNaSo12}):

\begin{itemize}
\item[(A)] Non-negativity: For all $i,k$,
	$$r_{k}^{i}(T)\geq0,r_{k}^{i+\frac{1}{2}}(T)\geq0.$$

\item[(B)] Shape constraints: For all $i$,
	$$\lambda_i + \sum_{k}r_{k}^{i}=\nu_i.$$

\item[(C)] Content constraints: For all $k$,
	$$\sum_{i}r_{k}^{i}+r_{k}^{i+\frac{1}{2}}=\mu_k.$$

\item[(D)] Gap constraints: For all $i,k$,
	$$r_{k}^{i+\frac{1}{2}}\leq \Big(\lambda_{i}+\sum_{k'<k}r_{k'}^{i}\Big) - \Big(\lambda_{i+1} + \sum_{k'\leq k}r_{k'}^{i+1}\Big).$$

\item[(E)] Too high: For all $i<k$,
	$$r_{k}^{i}(T)=0,r_{k}^{i+\frac{1}{2}}(T)=0.$$

\item[(F)] Reverse lattice word constraints: For all $i,k$,
	$$\sum_{i'<i} r_k^{i'} + r_k^{i'+\frac{1}{2}} \geq r_{k+1}^i + \sum_{i'<i}	r_{k+1}^{i'} + r_{k+1}^{i'+\frac{1}{2}}.$$
\end{itemize}

Define a polytope
\[{\mathcal P}_{\lambda,\mu}^{\nu}=\{(r_k^i, r_k^{i+\frac{1}{2}}): \text{(A)--(F)}\}\subseteq {\mathbb R}^{2l(\nu)l(\mu)}.\]

\begin{proposition}
\label{prop:finale}
${\tt EdgeTab}_{\lambda,\mu}^{\nu}\neq\emptyset\iff 
{\mathcal P}_{\lambda,\mu}^{\nu}\cap {\mathbb Z}^{2l(\nu)l(\mu)}\neq \emptyset$.
\end{proposition}

\begin{proof}
($\Rightarrow$) Let $T\in {\tt EdgeTab}_{\lambda,\mu}^{\nu}$. Clearly
$r_k^i(T)$ and $r_k^{i+\frac{1}{2}}(T)$ satisfy (A), (B), (C), and (E) above. The tableau constraint (D) asks that there to be enough edges in row $i+\frac{1}{2}$, between the rightmost $k$ in row $i+1$ and the leftmost $k$ in row $i$, to accommodate $r_k^{i+\frac{1}{2}}$ many $k$'s; this holds by semistandardness of $T$. 
Finally, (F) merely asks that the row word will be lattice after reading all the $k+1$'s in row $i$; this is certainly true of $T$ as it is row lattice. 

($\Leftarrow$) Let 
$(r_k^i, r_k^{i+\frac{1}{2}})\in {\mathcal P}_{\lambda,\mu}^{\nu}\cap {\mathbb Z}^{2l(\nu)l(\mu)}$.  
Construct a tableau $T^\star$ of shape $\nu/\lambda$ and content $\mu$ as follows. First, for all $i,k$, (uniquely) place $r_k^i$ many $k$'s in row $i$, such that the $k$'s are weakly increasing along each row. At this point the tableau has no edge labels but, by (B) has the correct skew shape $\nu/\lambda$. Moreover, (A) and (D) combined implies that for all $i,k$,
	\[\lambda_{i+1} + \sum_{k'\leq k}r_{k'}^{i+1}\leq \lambda_{i}+\sum_{k'<k}r_{k'}^{i}.\]
This precisely asserts that the partially
built up $T^\star$ is column strict.

Next, place $r_k^{i+\frac{1}{2}}$ many $k$'s as far to the right as possible in row $i+\frac{1}{2}$ without breaking the semistandardness of $T$. To be precise, the last $k$ will be in column $\lambda_{i}+\sum_{k'< k}r_{k'}^{i}$ and
the remaining $k$'s will be in adjacent columns to the left, namely columns:
\begin{equation}
\label{eqn:thecolsabc}
\left(\lambda_{i}+\sum_{k'< k}r_{k'}^{i}\right)-r_k^{i+\frac{1}{2}}+1,
\left(\lambda_{i}+\sum_{k'< k}r_{k'}^{i}\right)-r_k^{i+\frac{1}{2}}+2, \ldots,
\lambda_{i}+\sum_{k'< k}r_{k'}^{i}.
\end{equation}

This completes $T^\star$. 

(D) asserts that column strictness is preserved in the final step where we have added the edges: We are placing the $k$'s in row $i+\frac{1}{2}$ to the right of the box labels $\leq k$ in row $i+1$. Also, in row $i$, the columns
(\ref{eqn:thecolsabc}) contain box labels $<k$. Now,
(E) says no edge label is too high. 

However, (F) 
does not \textit{a priori} show $w_r(T^\star)$ is lattice (see Example~\ref{exa:0809} below). Thus we need:

\begin{claim}\label{claim:edgePermutability}
$w_r(T^\star)$ is lattice.
\end{claim}

\noindent\emph{Proof of claim:} Consider
row $i=1$. In this case (F) asserts
$r_{k+1}^1\leq 0$ for all $k\geq 1$. 
In view of (A), this means that there are
no labels greater than $1$ in the first row of $T^{\star}$. Moreover, if we know row latticeness has not failed before reading row $i>1$ then (F) immediately says no violation can occur in row $i$ either.

Therefore, it remains to check that while reading the edge labels in a row $i+\frac{1}{2}$ one always remains 
lattice. Suppose not. (F) combined with (A) implies
\begin{equation}
\label{eqn:says0810}
\sum_{i'<i+1} r_k^{i'} + r_k^{i'+\frac{1}{2}} \geq \sum_{i'<i+1}	r_{k+1}^{i'} + r_{k+1}^{i'+\frac{1}{2}}.
\end{equation}
That is, after reading the \emph{entirety} of row $i + \frac{1}{2}$, the row reading word has at least as many $k$'s as $k+1$'s.

Say edge $(i+\frac{1}{2},j)$ contains a violating label $k+1$ that breaks the latticeness of the row word. We may assume this $k+1$ is first (rightmost) among all such labels. By (\ref{eqn:says0810}), and/or the sentence immediately after it, there must be an ``extra'' edge label $k$ in row $i+\frac{1}{2}$ and in column $j$ or to its 
left. 

\noindent{\sf Case 1:} [$T_{i,j}<k$] The rightmostness of the placement of the $k$'s (see (\ref{eqn:thecolsabc})) implies that $k\in T_{i+\frac{1}{2},j}$. Hence, the row word has more $k+1$'s than $k$'s before reading edge $(i+\frac{1}{2},j)$, a contradiction
of the rightmostness of $k+1$.
That is, this case cannot actually occur.

\noindent{\sf Case 2:} [$T_{i,j}=k$] By semistandardness, for every $k+1$ in an edge $(i+\frac{1}{2},j')$ with $j' \ge j$, there is $k\in T_{i,j'}$. 
It is then straightforward to conclude there are more $k$'s than $k+1$'s before reading the edge $(i+\frac{1}{2},j)$, and in particular, before it even reads row $i$, a contradiction.\qed

This completes the proof of the
proposition.
\end{proof}

\begin{example}\label{exa:0809}
To illustrate the $(\Leftarrow)$ proof above,
let $\lambda=(2,2,1,1,0), \mu=(2,2,2,1,1)$, and $\nu=(2,2,2,2,2)$. Now, $r_1^3 = r_2^4 = r_1^{4+\frac{1}{2}} = r_2^{4+\frac{1}{2}} = r_4^{5+\frac{1}{2}} = r_5^{5+\frac{1}{2}} = 1$ and $r_3^5 = 2$ (all other components are zero) defines a lattice
point in ${\mathcal P}_{\lambda,\mu}^{\nu}$. There are four
edge labeled tableaux
that have these statistics, namely
\[\begin{picture}(350,90)
\put(0,70){$\tableau{{\ }&{\ }\\
{\ }&{\ }\\{\ }&{1 }\\{\ }&{2 }\\{3 }&{3 }}$}
\put(7,-9){$4$}
\put(26,-9){$5$}
\put(4,10){$12$}

\put(100,70){$\tableau{{ \ }&{\ }\\{ \ }&{ \ }\\{\ }&{1 }\\{\ }&{2 }\\{3 }&{3 }}$}
\put(107,-9){$5$}
\put(126,-9){$4$}
\put(104,10){$12$}

\put(200,70){$\tableau{{\ }&{\ }\\{\ }&{\ }\\{\ }&{1 }\\{\ }&{2 }\\{3 }&{3 }}$}
\put(210,-9){$5$}
\put(203,-9){$4$}
\put(204,10){$12$}

\put(300,70){$\tableau{{\ }&{\ }\\{\ }&{\ }\\{\ }&{1 }\\{\ }&{2 }\\{3 }&{3 }}$}
\put(330,-9){$5$}
\put(323,-9){$4$}
\put(304,10){$12$}
\end{picture}\]
\noindent
The first is not lattice, but the other three are. The rightmost of them is $T^{\star}$.\qed
\end{example}

\noindent\emph{Conclusion of proof of Theorem~\ref{thm:main}:} 
Notice that since (A)--(F) is linear and homogeneous in the components of $\lambda$, $\mu$ and $\nu$, it follows that 
${\mathcal P}_{N\lambda,N\mu}^{N\nu}=N{\mathcal P}_{\lambda,\mu}^{\nu}$. 
Thus,
in view of Proposition~\ref{prop:finale}, we have constructed the desired polytope ${\mathcal P}_{\lambda,\mu}^{\nu}$
alluded to in the introduction, which is the only missing component of the argument given there. \qed

Using indicator variables, one easily modifies the above argument to construct a polytope ${\mathcal Q}_{\lambda,\mu}^{\nu}$ whose lattice points are in bijection with the tableaux in
${\tt EdgeTab}_{\lambda,\mu}^{\nu}$. However this polytope does
not satisfy
$N{\mathcal Q}_{\lambda,\mu}^{\nu}={\mathcal Q}_{N\lambda,N\mu}^{N\nu}$. Counting lattice points of this polytope is not equivalent to computing $C_{\lambda,\mu}^{\nu}$ since one needs a weighted count based on ${\tt apwt}$.

\section{Schubert calculus and complexity theory} \label{section:Schubert}

The Littlewood-Richardson polynomials appear in the topic of equivariant Schubert calculus of Grassmannians. The use of ``equivariant'' refers to ``equivariant cohomology,'' a type of enriched cohomology theory. We refer the reader to \cite{Knutson.Tao:puzzles} for background.

Another enriched cohomology theory is $K$-theory
(i.e., the Grothendieck ring of algebraic vector bundles). 
There is a lattice rule \cite{Buch:Acta} for the corresponding $K$-theoretic Littlewood-Richardson coefficients $k_{\lambda,\mu}^{\nu}$ (another lattice rule uses \emph{genomic tableaux} \cite{Pechenik.Yong}). 

\begin{question}
Is the decision problem $k_{\lambda,\mu}^{\nu}=0$ in ${\sf P}$?
\end{question}

One cannot use the same general method of this paper to resolve
Question~1. To be precise,
in \cite[Section~7]{Buch:Acta} it is noted that (up to a sign convention) 
$k_{(1,0),(1,0)}^{(2,1)}=1$ but 
$k_{(2,0),(2,0)}^{(4,2)}=0$. That is, the saturation statement
\[
k_{\lambda,\mu}^{\nu}\neq 0
\implies
k_{N\lambda,N\mu}^{N\nu}\neq 0\]
is false (the truth of the converse is not known). 
Therefore, there cannot exist a polytope ${\mathcal K}_{\lambda,\mu}^{\nu}$ with the dilation property $N{\mathcal K}_{\lambda,\mu}^{\nu}={\mathcal K}_{N\lambda,N\mu}^{N\nu}$ crucial for the argument we use. 

Quantum cohomology of Grassmannians is another deformation of significant interest. Work of A.~Buch-A.~Kresch-K.~Purbhoo-H.~Tamvakis \cite{BKPT} established a combinatorial rule for the coefficients $d_{\lambda,\mu}^{\nu}$ of this case. They proved the two-step case of a 
\emph{puzzle} conjecture of A.~Knutson; see \cite{BKT}. Also, P.~Belkale \cite{Belkale} has established a quantum saturation property for these quantum Littlewood-Richardson coefficients. However, even from these results a solution to the following is not clear to us:
\begin{question}
Is the decision problem $d_{\lambda,\mu}^{\nu}=0$ in ${\sf P}$?
\end{question}

Since there are combinatorial rules for $k_{\lambda,\mu}^{\nu}$ and $d_{\lambda,\mu}^{\nu}$, these problems are in $\#{\sf P}$. They are both 
$\#{\sf P}$-complete problems because they contain as special cases the 
Littlewood-Richardson coefficients, which is $\#{\sf P}$-complete by \cite{Nara}.

Finally, the Schur polynomials are special cases of \emph{Schubert polynomials}
${\mathfrak S}_w(X)$, defined for any permutation $w\in S_n$. 
It is known that ${\mathfrak S}_{w'}={\mathfrak S}_{w}$ if $w'$ is the image of $w$ under 
the natural embedding of $S_n\hookrightarrow S_{n+1}$. Therefore it is unambiguous to refer to ${\mathfrak S}_w$ for $w\in S_{\infty}$. These polynomials form a ${\mathbb Z}$-linear basis of the ring of polynomials ${\mathbb Q}[x_1,x_2,\ldots]$. The structure constants $c_{u,v}^w$ relative to this basis are known to be nonnegative for Schubert calculus reasons. We refer to 
the book \cite{Manivel} for background and references.

It is a longstanding open problem to find a combinatorial rule for $c_{u,v}^w$. That is, it is not known if the problem is in $\#{\sf P}$. Since the Schubert structure constants also contain the Littlewood-Richardson coefficients in a specific way, 
H.~Narayanan's aforementioned theorem implies the problem is $\#{\sf P}$-hard. It should also not be difficult to show (using the Monk-Chevalley formula) that the Schubert problem is
in ${\sf GapP}$, that is, a difference of two $\#{\sf P}$ problems. 
Compare this with results of \cite{Burgisser}
for Kronecker coefficients.

\begin{question}
Is the decision problem  $c_{u,v}^w= 0$ {\sf NP}-hard?
\end{question}
Recently, it was shown that vanishing of Kronecker coefficients is {\sf NP}-hard \cite{Iken}. This establishes a formal difference in difficulty between the Kronecker coefficients and the Littlewood-Richardson coefficients. Inspired by their results, Question~3 asks if one can similarly establish a formal complexity difference in Schubert calculus. In addition, should either decision problem in Questions~1 or 2 be ${\sf NP}$-complete, one can conclude, in the Schubert calculus context, that such a formal difference does not preclude existence of a general combinatorial rule.

\section*{Acknowledgments}
We thank Christian Ikenmeyer for helpful exchanges.
AY was supported by an NSF grant.

\end{document}